\documentclass[12pt,a4paper,reqno,twoside]{amsart}

\usepackage[english]{babel}
\usepackage{stmaryrd}
\usepackage{dsfont}
\usepackage[symbol*,ragged]{footmisc}
\usepackage[colorlinks,linkcolor=red,anchorcolor=blue,citecolor=blue,urlcolor=blue]{hyperref}
\usepackage{color,xcolor}

\usepackage{geometry}
\usepackage{amssymb}
\usepackage{amsmath}
\usepackage{mathrsfs}
\usepackage{amsfonts}
\usepackage{epsfig}

\usepackage{amsthm}
\usepackage{amsxtra}
\usepackage{bbding}
\usepackage{epsfig}
\usepackage{graphicx}
\usepackage{latexsym}
\usepackage{mathbbol}
\usepackage{bbold}

\usepackage{pifont}
\usepackage{wasysym}
\usepackage{skull}
\usepackage{float}

\DeclareSymbolFontAlphabet{\mathbb}{AMSb}
\DeclareSymbolFontAlphabet{\mathbbol}{bbold}

\usepackage{amscd}
\usepackage[all]{xy}
\allowdisplaybreaks[4]
\usepackage{setspace}

\theoremstyle{plain}
\newtheorem{theorem}{Theorem}[section]
\newtheorem{proposition}{Proposition}[section]
\newtheorem{lemma}[proposition]{Lemma}

\newtheorem*{corollary*}{Corollary}

\theoremstyle{remark}

\numberwithin{equation}{section}
\addtocounter{footnote}{1}

\renewcommand{\footnoterule}{
  \kern -3pt
  \hrule width 2.5in height 0.4pt
  \kern 3pt
}

\makeatletter
\@ifundefined{MakeUppercase}{}{}
\makeatother

\begin{document}
	
\title[Primes in the intersection of two Piatetski--Shapiro sets]
	  {Primes in the intersection of two Piatetski--Shapiro sets}

\author[Xiaotian Li, Wenguang Zhai, Jinjiang Li]{Xiaotian Li \quad \& \quad Wenguang Zhai \quad \& \quad Jinjiang Li}

\address{(Corresponding author) Department of Mathematics, China University of Mining and Technology,
         Beijing, 100083, People's Republic of China}

\email{xiaotian.li.math@gmail.com}


\address{Department of Mathematics, China University of Mining and Technology,
         Beijing, 100083, People's Republic of China}

\email{zhaiwg@hotmail.com}

\address{Department of Mathematics, China University of Mining and Technology,
         Beijing 100083, People's Republic of China}

\email{jinjiang.li.math@gmail.com}

\date{}

\footnotetext[1]{Xiaotian Li is the corresponding author.  

\quad\,\,
{\textbf{MR(2020) Subject Classification}}: 11N05, 11N80, 11L07, 11L20

}

\begin{abstract}
Let $\pi(x;\gamma_1,\gamma_2)$ denote the number of primes $p$ with $p\leqslant x$ and $p=\lfloor n^{1/\gamma_1}_1\rfloor=\lfloor n^{1/\gamma_2}_2\rfloor$, where $\lfloor t\rfloor$ denotes the integer part of $t\in\mathbb{R}$ and $1/2<\gamma_2<\gamma_1<1$ are fixed constants. In this paper, we show that $\pi(x;\gamma_1,\gamma_2)$ holds an asymptotic formula for $21/11<\gamma_1+\gamma_2<2$, which constitutes an improvement upon the previous result of Baker \cite{Baker-2014}.
\end{abstract}

\maketitle

\section{Introduction and main result}
Let $1/2<\gamma<1$ be a fixed real number and $F_{\gamma}:=\{\lfloor n^{1/\gamma}\rfloor: n\in \mathbb{N}^+\}.$
In 1953, Piatetski--Shapiro \cite{P-S-1953} first studied the distribution of primes in the sequence $F_{\gamma}$. Let
\begin{equation*}
  \pi_{\gamma}(x):=\sum_{\substack{p\leqslant x\\p=\lfloor n^{1/\gamma}\rfloor}}1.
\end{equation*}
Piatetski--Shapiro \cite{P-S-1953} proved that, for $11/12<\gamma<1$, the asymptotic formula
\begin{equation}\label{P-S-PNT}
\pi_{\gamma}(x)=\frac{x^{\gamma}}{\log x}(1+o(1))
\end{equation}
holds as $x\to \infty$. The range of $\gamma$, which makes (\ref{P-S-PNT}) hold, was enlarged to $9/10$, $62/69$, $662/755$, $34/39$, $13/15$, $5302/6121$, $2426/2817$ by Kolesnik \cite{Kolesnik-1967-1972}, Graham (unpublished) and Leitmann \cite{Leitamnn-1980}, Heath--Brown \cite{Heath-Brown-1983}, Kolesnik \cite{Kolesnik-1985}, Liu and Rivat \cite{Liu-Rivat-1992}, Rivat \cite{Rivat-1992}, Rivat and Sargos \cite{Rivat-Sargos-2001}, respectively. The asymptotic formula (\ref{P-S-PNT}) is called the  Piatetski--Shapiro prime number theorem.

In 1982, Leitmann \cite{Leitmann-1982} studied a binary analogue of the Piatetski--Shapiro prime number theorem. Suppose that $1/2<\gamma_2<\gamma_1<1$ are two fixed constants. Define
\begin{equation*}
\pi(x;\gamma_1,\gamma_2):=\sum_{\substack{p\leqslant x\\ p=\lfloor n^{1/\gamma_1}_1\rfloor=\lfloor n^{1/\gamma_2}_2\rfloor}}1.
\end{equation*}
Leitmann \cite{Leitmann-1982} proved that, for $55/28<\gamma_1+\gamma_2<2$, the asymptotic formula
\begin{equation}\label{Intersection-P-S-PNT}
  \pi(x;\gamma_1,\gamma_2)=\frac{\gamma_1\gamma_2}{\gamma_1+\gamma_2-1}\cdot
  \frac{x^{\gamma_1+\gamma_2-1}}{\log x}(1+o(1))
\end{equation}
holds as $x\to \infty$, which is a generalization and analogue of (\ref{P-S-PNT}). In 1983, Sirota \cite{Sirota-1983} proved that (\ref{Intersection-P-S-PNT}) holds for $31/16<\gamma_1+\gamma_2<2$. Afterwards, Baker \cite{Baker-2014} improved the result of Sirota \cite{Sirota-1983}, and showed that $23/12<\gamma_1+\gamma_2<2$.

In this paper, we shall continue to improve the result of Baker \cite{Baker-2014}, and establish the following theorem.

\begin{theorem}\label{Theorem}
Suppose that $1/2<\gamma_2<\gamma_1<1$ with $21/11<\gamma_1+\gamma_2<2$. Then there holds
\begin{equation*}
\pi(x;\gamma_1,\gamma_2)=\gamma_1\gamma_2\int_{2}^{x}\frac{t^{\gamma_1+\gamma_2-2}}{\log t}\mathrm{d}t+O\big(x^{\gamma_1+\gamma_2-1}e^{-c_0\sqrt{\log x}}\big),
\end{equation*}
where $c_0>0$ is an absolute constant.
\end{theorem}

\textbf {Remark.} Baker \cite{Baker-2014} and  Zhai \cite{Zhai-1999} studied the distribution of primes in the intersection of $F_{\gamma_1}, \dots, F_{\gamma_k},$ where
$k\geqslant 3$ is a fixed integer.

\smallskip
\textbf{Notation.}
Throughout this paper, $\varepsilon$ and $\eta$ are sufficiently small positive numbers, which may be different in each occurrences. Let $p$, with or without subscripts, always denote a prime number. We use $\lfloor t\rfloor,\{t\}$ and $\|t\|$ to denote the integral part of $t$, the fractional part of $t$ and the distance from $t$ to the nearest integer, respectively. As usual, $\Lambda(n)$ and $d(n)$ denote von Mangoldt's function and Dirichlet divisor function, respectively. We write $\psi(t)=t-\lfloor t\rfloor-1/2,e(x)=e^{2\pi ix},\mathscr{L}=\log x$. The notation $n\sim N$ means $N<n\leqslant 2N$ and  $f(x)\ll g(x)$ means  $f(x)=O(g(x))$.  Denote by $c_j$ the positive constants which depend at most on $\gamma_1$ and $\gamma_2$. We use $\mathbb{N}^+$ and $\mathbb{Z}$ to denote the set of positive natural numbers and the set of integers, respectively.

\section{ Preliminaries}

\noindent
In this section, we we list some lemmas which are necessary for proving Theorem \ref{Theorem}.

\begin{lemma}\label{1-2-derivative}
Suppose that $5<a<b\leqslant2a$ and  $f''\in C[a,b]$. If $0<c_1\lambda_1\leqslant |f'(x)|\leqslant c_2\lambda_1$, then
\begin{equation}\label{1-derivative}
\sum_{a<n\leqslant b}e(f(n))\ll \lambda^{-1}_1.
\end{equation}
If $0<c_3\lambda_2\leqslant|f''(x)|\leqslant c_4\lambda_2$, then
\begin{equation}\label{2-derivative}
\sum_{a<n\leqslant b}e(f(n))\ll a\lambda^{1/2}_2+\lambda^{-1/2}_2.
\end{equation}
where $c_i\,(i=1,2,3,4)$ are absolute constants.
\end{lemma}
\begin{proof}
See Lemma $2.1$ of Graham and Kolesnik \cite{Graham-Kolesnik-1991}.  $\hfill$
\end{proof}

Let $k\geqslant 2$ be a fixed integer. Suppose that $a_1,\dots,a_k$ are real numbers with $a_1\cdots a_k\neq0$. Let $\alpha_1,\alpha_2,\dots,\alpha_k$ be distinct real constants such that $\alpha_j\notin \mathbb{Z}\,(j=1,\dots,k)$. $u_1,\dots,u_k$ are real numbers with $u_j\in[0,1]\,(j=1,\dots,k)$. $M$ and $M_1$ are  large positive real numbers with $M<M_1\leqslant 2M$. For convenience, we set
\begin{align*}
 & f_k(m)=a_1(m+u_1)^{\alpha_1}+\cdots+a_k(m+u_k)^{\alpha_k},\\
 & R=|a_1|M^{\alpha_1}+\dots+|a_k|M^{\alpha_k},\qquad
   S_k(M):=\sum_{M<m\leqslant M_1}e(f_k(m)).
\end{align*}
\begin{lemma}\label{S_k(M)-Zhai}
Suppose that $\eta$ is a sufficiently small fixed positive number.
If $R\leqslant \eta M$, then
\begin{equation}\label{1-derivative-estimate}
S_k(M)\ll MR^{-1/k}.
\end{equation}
If $R\gg M$, then
\begin{equation}\label{2-derivative-estimate}
S_k(M)\ll R^{1/2}+MR^{-1/(k+1)}.
\end{equation}
\end{lemma}
\begin{proof}
See Proposition $2$ and Proposition $3$  of Zhai \cite{Zhai-1999}.  $\hfill$
\end{proof}

\begin{lemma}\label{k-min-esti}
Suppose that $1/2<\gamma_k<\dots<\gamma_1<1$ are fixed real numbers. Let
\begin{equation*}
S(M;\gamma_1,\dots,\gamma_k):=\sup_{(u_1,\dots,u_k)\in [0,1]^k}\sum_{M<m\leqslant2M}\prod_{j=1}^k\min\left(1,\frac{1}{H_j\|(m+u_j)^{\gamma_j}\|}\right),
\end{equation*}
where $H_j>1\,(j=1,\dots,k)$ are real numbers. If $\gamma_1+\dots+\gamma_k>k-\frac{1}{k+1}$, then
\begin{equation*}
S(M;\gamma_1,\dots,\gamma_k)\ll M(H_1\cdots H_k)^{-1}(\log M)^k+M^{\frac{k}{k+1}}(\log M)^{k}.
\end{equation*}
\end{lemma}
\begin{proof}
See Proposition $4$ of Zhai \cite{Zhai-1999}.  $\hfill$
\end{proof}

\begin{lemma}\label{Finite-Fourier-expansion}
For any $H>1$, we have
\begin{equation*}
  \psi(\theta)=-\sum_{0<|h|\leqslant H}\frac{e(\theta h)}{2\pi ih}+O\left(g(\theta,H)\right),
\end{equation*}
where
\begin{equation*}
g(\theta,H)=\min\left(1,\frac{1}{H\|\theta\|}\right)
=\sum_{h=-\infty}^{\infty}a(h)e(\theta h),
\end{equation*}
\begin{equation*}
a(0)\ll \frac{\log 2H}{H},\ \
a(h)\ll \min\left(\frac{1}{|h|},\frac{H}{h^2}\right)\ (h\neq 0).
\end{equation*}
\end{lemma}
\begin{proof}
See Section $2$ of Heath--Brown \cite{Heath-Brown-1983}.  $\hfill$
\end{proof}

\begin{lemma}\label{Weyl inequality}
Let $z(n)$ be any complex numbers, $1\leqslant Q\leqslant N$. Then we have
\begin{equation*}
  \left|\sum_{N<n\leqslant CN}z(n)\right|^2\ll \frac{N}{Q}\sum_{|q|\leqslant Q}\left(1-\frac{|q|}{Q}\right)\Re\sum_{N<n\leqslant CN-q}z(n)\overline{z(n+q)}.
\end{equation*}
\end{lemma}
\begin{proof}
See Lemma 2.5 of Graham and Kolesnik \cite{Graham-Kolesnik-1991}.  $\hfill$
\end{proof}

\begin{lemma}\label{Inversion-Formula}
Suppose that $f^{(4)}\in C[a,b]$ and $g''\in C[a,b]$  such that
\begin{align*}
&|f''(x)|\sim\frac{1}{D}, \qquad|f^{(3)}(x)|\ll\frac{1}{DU}\,(U\geqslant1),\\
&|g(x)|\ll G,
\qquad|g'(x)|\ll\frac{G}{U_1}\,(U_1\geqslant1).
\end{align*}
Then one has
\begin{align*}
\sum_{a<n\leqslant b}g(n)e\left(f(n)\right)&=\sum_{\alpha<v\leqslant \beta}\frac{b_vg(n_v)}{\sqrt{|f''(n_v)|}}e\left(f(n_v)-vn_v+\frac{1}{8}\right)\\
&\quad+O\left(G\log(\beta-\alpha+2)+U^{-1}G(b-a+D)\right)\\
&\quad+O\left(G\min\left(\sqrt{D},\frac{1}{\|\alpha\|}\right)
+G\min\left(\sqrt{D},\frac{1}{\|\beta\|}\right)\right),
\end{align*}
where $[\alpha,\beta]$ is the image of $[a,b]$ under the mapping $y=f'(x)$, $n_v$ is the solution of the equation $f'(x)=v,$ and
\begin{equation*}
b_v=
\begin{cases}
     1,  & \textrm{if $\alpha<v<\beta$},\\
    \frac{1}{2}, & \textrm{if $u=\alpha\in \mathbb{Z}$ or $u=\beta\in \mathbb{Z}$}.
   \end{cases}
\end{equation*}
\end{lemma}
\begin{proof}
See Theorem 3.1 Karatsuba and Voronin \cite{Karatsuba-Voronin-1992}.
\end{proof}

\begin{lemma}\label{Estimate-min}
Suppose that $f(n)\ll \mathcal{Q},\,f'(n)\gg \Delta$ for $n\sim N$. Then we have
\begin{equation*}
  \sum_{n\sim N}\min \left(\mathcal{D},\frac{1}{\|f(n)\|}\right)\ll(\mathcal{Q}+1)(\mathcal{D}+\Delta^{-1})\log(2+\Delta^{-1}).
\end{equation*}
\end{lemma}
\begin{proof}
See Lemma 2.8 of Kr\"{a}tzel \cite{Kratzel-1988}.  $\hfill$
\end{proof}

\section{ Exponential Sum Estimate Over Primes}

Suppose that $1/2<\gamma_2<\gamma_1<1,\,\,m\sim M,\,\,n\sim N,\,\,mn\sim X$. $h_1,h_2$ are integers satisfying  $1\leqslant |h_1|\leqslant X^{1-\gamma_1+\varepsilon},\,\,1\leqslant |h_2|\leqslant X^{1-\gamma_2+\varepsilon}$. Define
\begin{align*}
S_{I}(M,N) := & \,\, \sum_{M<m\leqslant2M}a(m)\sum_{N<n\leqslant2N}e\left(h_1(mn)^{\gamma_1}
+h_2(mn)^{\gamma_2}\right),\\
S_{II}(M,N):= & \,\,\sum_{M<m\leqslant2M}a(m)\sum_{N<n\leqslant2N}b(n)e\left(h_1(mn)^{\gamma_1}
+h_2(mn)^{\gamma_2}\right).
\end{align*}
where $a(m)$ and $b(n)$ are complex numbers satisfying $a(m)\ll1,\,\,b(n)\ll1$. For convenience, we write $\mathcal{R}=|h_1|X^{\gamma_1}+|h_2|X^{\gamma_2}$. Trivially, there holds $X^{\gamma_1}\ll\mathcal{R}\ll X^{1+\varepsilon}$.
\begin{lemma}\label{S(II)}
Suppose that $a(m)\ll 1,\,\,b(n)\ll 1$. Then, for $X^{2/11+8\varepsilon}\ll N\ll X^{16/11-24\varepsilon}\mathcal{R}^{-1}$,
we have
\begin{equation*}
S_{II}(M,N)\ll X^{\gamma_1+\gamma_2-1-4\varepsilon}.
\end{equation*}
\end{lemma}
\begin{proof}
Let $Q=\lfloor X^{2/11+8\varepsilon}\rfloor=o(N)$. By Cauchy's inequality and Lemma \ref{Weyl inequality}, we have
\begin{align*}
|S_{II}(M,N)|^2
& \ll M\sum_{M<m\leqslant2M}\left|\sum_{N<n\leqslant2N}b(n)e\left(h_1(mn)^{\gamma_1}+h_2(mn)^{\gamma_2}\right)\right|^2\\
& \ll \frac{M^2N^2}{Q}+\frac{MN}{Q}\sum_{1\leqslant q\leqslant Q}\sum_{N<n\leqslant2N-q}\left|\sum_{M<m\leqslant2M}e\left(f(m,n)\right)\right|,
\end{align*}
where
\begin{align}
& f(m,n)=h_1m^{\gamma_1}\Delta(n,q;\gamma_1)+h_2m^{\gamma_2}\Delta(n,q;\gamma_2),\label{f(m,n)}\\
& \Delta(n,q;\gamma_i)=(n+q)^{\gamma_i}-n^{\gamma_i}(i=1,2).\label{Delta(n,q;gamma)}
\end{align}
Taking $k=2,\,(a_1,a_2,u_1,u_2,\alpha_1,\alpha_2)
=\left(h_1\Delta(n,q;\gamma_1),h_2\Delta(n,q;\gamma_2),0,0,\gamma_1,\gamma_2\right)$ in Lemma \ref{S_k(M)-Zhai}, it is easy to see that $R=qN^{-1}\mathcal{R}$, which combined with (\ref{2-derivative-estimate}) of Lemma \ref{S_k(M)-Zhai} yields
\begin{align*}
    & \,\, \sum_{1\leqslant q\leqslant Q}
           \sum_{N<n\leqslant2N}\left(\left(\frac{q\mathcal{R}}{N}\right)^{1/2}
           +M\left(\frac{q\mathcal{R}}{N}\right)^{-1/3}\right)
                  \nonumber \\
\ll & \,\, N^{1/2}Q^{3/2}\mathcal{R}^{1/2}+MN^{4/3}Q^{2/3}\mathcal{R}^{-1/3}
           \ll  X+X^{53/33-8\varepsilon/3}R^{-2/3}\ll X.
\end{align*}
This completes the proof of Lemma \ref{S(II)}. $\hfill$
\end{proof}

\begin{lemma}\label{S(I)}
Suppose that $a(m)\ll1$, $M\ll X^{17/11-21\varepsilon}\mathcal{R}^{-1}$.
Then we have
\begin{equation*}
S_I(M,N)\ll X^{\gamma_1+\gamma_2-1-4\varepsilon}.
\end{equation*}
\end{lemma}
\begin{proof}
Taking $k=2,\,\,(a_1,a_2,u_1,u_2,\alpha_1,\alpha_2)
=\left(h_1m^{\gamma_1},h_2m^{\gamma_2},0,0,\gamma_1,\gamma_2\right)$ in Lemma \ref{S_k(M)-Zhai}, it is easy to see that $R=\mathcal{R}$. If $M\ll X^{10/11-4\varepsilon}\mathcal{R}^{-1/2}$, by (\ref{2-derivative-estimate}) of Lemma \ref{S_k(M)-Zhai}, we deduce that
\begin{align*}
S_I(M,N)
& \ll\sum_{M<m\leqslant2M}\left|\sum_{N<n\leqslant2N}e\left(h_1(mn)^{\gamma_1}+h_2(mn)^{\gamma_2}\right)\right|\\
& \ll \sum_{M<m\leqslant2M}\left(\mathcal{R}^{1/2}+N\mathcal{R}^{-1/3}\right)
\ll M\mathcal{R}^{1/2}+MN\mathcal{R}^{-1/3}
\ll X^{\gamma_1+\gamma_2-1-4\varepsilon}.
\end{align*}
Later, we always suppose that $M\gg X^{10/11-4\varepsilon}\mathcal{R}^{-1/2}$. Taking $Q=\lfloor X^{2/11+8\varepsilon}\rfloor$, by Cauchy's inequality and Lemma \ref{Weyl inequality}, we obtain
\begin{align*}
|S_I(M,N)|^2
& \ll M\sum_{M<m\leqslant2M}\left|\sum_{N<n\leqslant2N}b(n)\overline{b}(n+1)e\left(f(m,n)\right)\right|^2\\
& \ll \frac{M^2N^2}{Q}+\frac{MN}{Q}\sum_{1\leqslant q\leqslant Q}\left|\sum_{M<m\leqslant2M}\sum_{N<n\leqslant2N-q}e\left(f(m,n)\right)\right|,
\end{align*}
where $f(m,n)$ is defined by (\ref{f(m,n)}). It suffices to show that
\begin{equation}\label{S(X,Q)}
\mathcal{S}(X,Q):=\sum_{1\leqslant q\leqslant Q}\left|\sum_{M<m\leqslant2M}\sum_{N<n\leqslant2N-q}e\left(f(m,n)\right)\right|\ll X.
\end{equation}
Next, we shall prove (\ref{S(X,Q)}) in five cases.

\noindent
\textbf{Case 0.}
For $a,b>0$, let $N(a,b)$ denote the number of solutions of the inequality
\begin{equation}\label{case 0}
  |ah_1(mn)^{\gamma_1}+bh_2(mn)^{\gamma_2} |\leqslant \frac{\mathcal{R}}{Q^{1/2}\mathscr{L}},
\end{equation}
with $m\sim M, n\sim N$. Suppose that $\eta\in(0,1)$ is a sufficiently small positive constant. Then we can prove
that
\begin{equation}\label{N(a,b)}
  N(a,b)\ll_{\eta} X^{\gamma_1+\gamma_2-1-4\varepsilon}
\end{equation}
holds for $a,b\in [\eta,1/\eta]$ uniformly. If $h_1h_2>0$, then $N(a,b)=0$.
Now we assume that $h_1h_2<0$. If $(m,n)$ satisfies the inequality (\ref{case 0}), then
\begin{align*}
ah_1(mn)^{\gamma_1}
& =-bh_2(mn)^{\gamma_2}+O\left(\frac{\mathcal{R}}{Q^{1/2}\mathscr{L}}\right)
 =-bh_2(mn)^{\gamma_2}\left(1+O\left(\frac{1}{Q^{1/2}\mathscr{L}}\right)\right),
\end{align*}
which implies that
\begin{align*}
mn & = \left(-\frac{bh_2}{ah_1}\right)^{\frac{1}{\gamma_1-\gamma_2}}
     \left(1+O\left(\frac{1}{Q^{1/2}\mathscr{L}}\right)\right)
      =\left(-\frac{bh_2}{ah_1}\right)^{\frac{1}{\gamma_1-\gamma_2}}+O\left(\frac{X}{Q^{1/2}\mathscr{L}}\right).
\end{align*}
Thus (\ref{N(a,b)}) follows from a divisor argument. We shall explain the reason why we investigate this case in Case 4.

\noindent
\textbf{Case 1.} If $\left|\frac{\partial f(m,n)}{\partial m}\right|\leqslant \frac{1}{500}$, taking  $k=2$ and $$\left(a_1,a_2,u_1,u_2,\alpha_1,\alpha_2\right)
=(h_1\Delta(n,q;\gamma_1),h_2\Delta(n,q;\gamma_2),0,0,\gamma_1,\gamma_2)$$
in Lemma \ref{S_k(M)-Zhai}, then we have $R=qN^{-1}\mathcal{R}$. By (\ref{1-derivative-estimate}) of Lemma \ref{S_k(M)-Zhai}, we get
\begin{equation*}
\mathcal{S}(X,Q)\ll \sum_{1\leqslant q\leqslant Q}\sum_{N<n\leqslant2N-q}M\left(\frac{q\mathcal{R}}{N}\right)^{-1/2}
\ll MN^{3/2}Q^{1/2}\mathcal{R}^{-1/2}\ll X^{25/22+2\varepsilon}\mathcal{R}^{-1/4}\ll X.
\end{equation*}

\noindent
\textbf{Case 2.} If $\left|\frac{\partial f(m,n)}{\partial n}\right|\leqslant \frac{1}{500}$, it is easy to see that
\begin{equation*}
\frac{\partial f(m,n)}{\partial n}=\gamma_1h_1m^{\gamma_1}\Delta(n,q;\gamma_1-1)+\gamma_2h_2m^{\gamma_2}\Delta(n,q;\gamma_2-1),
\end{equation*}
where
\begin{equation}\label{Delta-n-q-gamma}
\Delta(n,q;\gamma_i-1)=(n+q)^{\gamma_i-1}-n^{\gamma_i-1}=(\gamma_i-1) qn^{\gamma_i-2}+O(q^2n^{\gamma_i-3}),\quad(i=1,2).
\end{equation}
Thus, we obtain
\begin{equation*}
\frac{\partial f(m,n)}{\partial n}=\gamma_1(\gamma_1-1)h_1m^{\gamma_1}qn^{\gamma_1-2}
+\gamma_2(\gamma_2-1)h_2m^{\gamma_2}qn^{\gamma_2-2}
+O\left(\frac{q^2}{N^3}\mathcal{R}\right).
\end{equation*}
If $h_1h_2>0$, we have $\left|\frac{\partial f(m,n)}{\partial n}\right|\sim qN^{-2}\mathcal{R}<1/2$. Thus by (\ref{1-derivative}) of Lemma \ref{1-2-derivative}, we derive that
\begin{equation}\label{Case2-h1h2>0}
\mathcal{S}(X,Q)
\ll\sum_{1\leqslant q\leqslant Q}\sum_{M<m\leqslant 2M}N^2q^{-1}\mathcal{R}^{-1}
\ll MN^2\mathcal{R}^{-1}\mathscr{L}\ll X^{12/11+5\varepsilon}\mathcal{R}^{-1/2}\ll X.
\end{equation}
Now we consider $h_1h_2<0$. Let $\delta=(Q\mathcal{R})^{1/2}N^{-3/2}\mathscr{L}^{-1/2}$. Define
\begin{align*}
 I_1=\left\{n:n\in[N,2N-q],\,\left|\frac{\partial f(m,n)}{\partial n}\right|\leqslant \delta\right\},\,\,
 I_2=\left\{n:n\in[N,2N-q],\,\left|\frac{\partial f(m,n)}{\partial n}\right|> \delta\right\} .
\end{align*}
If $n\in I_1$, then we have
\begin{align*}
\gamma_1(\gamma_1-1)h_1m^{\gamma_1}qn^{\gamma_1-2}
& =-\gamma_2(\gamma_2-1)h_2m^{\gamma_2}qn^{\gamma_2-2}
+O\left(\delta+\frac{q^2}{N^3}\mathcal{R}\right)\\
& =-\gamma_2(\gamma_2-1)h_2m^{\gamma_2}qn^{\gamma_2-2}
\left(1+O\left(\frac{\delta N^2}{q\mathcal{R}}+\frac{q}{N}\right)\right),
\end{align*}
which implies that
\begin{align*}
n & =\left(-\frac{\gamma_2(\gamma_2-1)h_2m^{\gamma_2}}
{\gamma_1(\gamma_1-1)h_1m^{\gamma_1}}\right)^{\frac{1}{\gamma_1-\gamma_2}}
\left(1+O\left(\frac{\delta N^2}{q\mathcal{R}}+\frac{q}{N}\right)\right)\\
 & =\left(-\frac{\gamma_2(\gamma_2-1)h_2m^{\gamma_2}}
{\gamma_1(\gamma_1-1)h_1m^{\gamma_1}}\right)^{\frac{1}{\gamma_1-\gamma_2}}
+O\left(\frac{\delta N^3}{q\mathcal{R}}+q\right).
\end{align*}
By (\ref{1-derivative}) of Lemma \ref{1-2-derivative}, one has
\begin{align}\label{Case2-h1h2<0}
           \mathcal{S}(X,Q)
\ll & \,\, \sum_{1\leqslant q\leqslant Q}\sum_{M<m\leqslant2M}\left(\sum_{n\in I_1}e\left(f(m,n)\right)
           +\sum_{n\in I_2}e\left(f(m,n)\right)\right)
                  \nonumber\\
\ll & \,\, \sum_{1\leqslant q\leqslant Q}\sum_{M<m\leqslant2M}\left(\frac{\delta N^3}{q\mathcal{R}}
             +q+\frac{1}{\delta}\right)
\ll  MN^{3/2}Q^{1/2}\mathcal{R}^{-1/2}\mathscr{L}^{1/2}+MQ^2
                  \nonumber\\
\ll & \,\,X^{25/22+7\varepsilon}\mathcal{R}^{-1/4}+X^{21/11-5\varepsilon}\mathcal{R}^{-1}\ll X.
\end{align}
Combining (\ref{Case2-h1h2>0}) and (\ref{Case2-h1h2<0}), we complete the estimate of Case $2$.

\noindent
\textbf{Case 3.}
For some $i,\,j$ satisfying $2\leqslant i+j\leqslant3$ and
\begin{equation}\label{mix-derivative}
\left|\frac{\partial^{i+j}f(m,n)}{\partial m^i\partial n^j}\right|\leqslant \frac{q\mathcal{R}}{QM^iN^{j+1}},
\end{equation}
taking $c(\gamma,0)=1,\,c(\gamma,j)=\gamma(\gamma-1)\cdots(\gamma+j-1)$ for $j\neq 0$, one has
\begin{equation*}
\frac{\partial^{i+j}f(m,n)}{\partial m^in^j}=c(\gamma_1,i)c(\gamma_1,j)h_1m^{\gamma_1-i}\Delta(n,q;\gamma_1-j)
+c(\gamma_2,i)c(\gamma_2,j)h_2m^{\gamma_2-i}\Delta(n,q;\gamma_2-j).
\end{equation*}
Since $c(\gamma_1,i)c(\gamma_1,j) $ and $c(\gamma_2,i)c(\gamma_2,j)$ always have the same sign, we may postulate that $h_1h_2<0$, otherwise there exists no $(m,n)$ satisfying (\ref{mix-derivative}). If $(m,n)$ satisfies (\ref{mix-derivative}), then we have
\begin{align*}
&c(\gamma_1,i)c(\gamma_1,j)h_1m^{\gamma_1-i}\Delta(n,q;\gamma_1-j)\\
=& -c(\gamma_2,i)c(\gamma_2,j)h_2m^{\gamma_2-i}\Delta(n,q;\gamma_2-j)+O\left(\frac{q\mathcal{R}}{QM^iN^{j+1}}\right)\\
=& -c(\gamma_2,i)c(\gamma_2,j)h_2m^{\gamma_2-i}\Delta(n,q;\gamma_2-j)\left(1+O\left(\frac{1}{Q}\right)\right),
\end{align*}
which implies that
\begin{align*}
m & =\left(-\frac{c(\gamma_2,i)c(\gamma_2,j)h_2\Delta(n,q;\gamma_2-j)}
{c(\gamma_1,i)c(\gamma_1,j)h_1\Delta(n,q;\gamma_1-j)}\right)^{\frac{1}{\gamma_1-\gamma_2}}
\left(1+O\left(\frac{1}{Q}\right)\right)\\
 & =\left(-\frac{c(\gamma_2,i)c(\gamma_2,j)h_2\Delta(n,q;\gamma_2-j)}
{c(\gamma_1,i)c(\gamma_1,j)h_1\Delta(n,q;\gamma_1-j)}\right)^{\frac{1}{\gamma_1-\gamma_2}}
+O\left(\frac{M}{Q}\right).
\end{align*}
Therefore, we have
\begin{equation*}
\mathcal{S}(X,Q)\ll\sum_{1\leqslant q\leqslant Q}\sum_{N<n\leqslant2N-q}\frac{M}{Q}\ll X.
\end{equation*}

\noindent
\textbf{Case 4}. We postulate that all the conditions from Case $0$ to Case $3$ do not hold. Without loss of generality, we assume that $\frac{\partial f(m,n)}{\partial n}>0$. Define
\begin{equation*} \mathcal{I}_j:=[\mathcal{A}_j(m),\mathcal{B}_j(m)]=\left\{n:n\in[N,2n-q],\,\frac{2^jq\mathcal{R}}{QN^3}<\left|\frac{\partial^2f(m,n)}{\partial n^2}\right|\leqslant \frac{2^{j+1}q\mathcal{R}}{QN^3}\ll \frac{q\mathcal{R}}{N^3}\right\},
\end{equation*}
\begin{equation*}
0\leqslant j\leqslant J_0:=\frac{\log Q}{\log 2}.
\end{equation*}
For convenience, we write $\frac{\partial f(m,n)}{\partial n}=f_{n}(m,n)$ and similar to other derivatives. If $n\in \mathcal{I}_j$, by Lemma \ref{Inversion-Formula}, we get
\begin{equation*}
\sum_{n\in\mathcal{I}_j}e\left(f(m,n)\right)=\sum_{v_1(m)<v\leqslant v_2(m)}b_v\frac{e\left(s(m,v)+1/8\right)}{\sqrt{|G(m,v)|}}+O\left(\mathcal{E}(m,q,j)\right),
\end{equation*}
where $v_1(m),\,v_2(m)$ are the images of $n\in \mathcal{I}_j$ under the mapping $y=f_n(m,n)$; $g:=g(m,v)$ is the solution of the equation $f_n(m,n)=v$; $s(m,v):=f\left(m,g(m,v)\right)-vg(m,v)$;
$G(m,v)=f_{nn}\left(m,g(m,v)\right)$ and
\begin{align*}
         \mathcal{E}(m,q,j)
= & \,\, \log X+\left(N+\frac{QN^3}{2^jq\mathcal{R}}\right)\frac{1}{N}
         +\min\left(\frac{Q^{1/2}N^{3/2}}{2^{j/2}q^{1/2}\mathcal{R}^{1/2}},\frac{1}{\|v_1(m)\|}\right)
                \nonumber\\
   & \,\,+\min\left(\frac{Q^{1/2}N^{3/2}}{2^{j/2}q^{1/2}\mathcal{R}^{1/2}},\frac{1}{\|v_2(m)\|}\right).
\end{align*}
It is easy to see that
\begin{align*}
 &1\ll\frac{q\mathcal{R}}{N^2}\ll v_1(m),v_2(m)\ll \frac{q\mathcal{R}}{N^2},\\
 & v'_1(m)=f_{nm}(m,\mathcal{B}_j)\gg\frac{q\mathcal{R}}{MN^2Q}, \qquad
   v'_2(m)=f_{nm}(m,\mathcal{A}_j)\gg\frac{q\mathcal{R}}{MN^2Q}.
\end{align*}
Taking $\mathcal{D}=(N^3Q)^{1/2}(2^jq\mathcal{R})^{-1/2},\mathcal{Q}=q\mathcal{R}N^{-2},
\Delta=q\mathcal{R}(MN^2Q)^{-1}$ in Lemma \ref{Estimate-min}, it is easy to see that the contribution of the term $\mathcal{E}(m,q,j)$ to $\mathcal{S}(X,Q)$ is
\begin{align*}
\ll & \,\,\sum_{1\leqslant q\leqslant Q}\sum_{1\leqslant j\leqslant J_0}\sum_{M<m\leqslant2M}\mathcal{E}(m,q,j)
              \nonumber \\
\ll & \,\,\sum_{1\leqslant q\leqslant Q}\sum_{1\leqslant j\leqslant J_0}\left(M\log X
      +\frac{MN^2Q}{2^jq\mathcal{R}}+
      \frac{q\mathcal{R}}{N^2}\cdot\frac{N^{3/2}Q^{1/2}}{2^{j/2}q^{1/2}\mathcal{R}^{1/2}}
      +\frac{q\mathcal{R}}{N^2}\cdot\left(\frac{q\mathcal{R}}{MN^2Q}\right)^{-1}\right)
              \nonumber \\
\ll & \,\,MQ\mathscr{L}^2+\frac{MN^2Q\mathscr{L}}{\mathcal{R}}+\frac{Q^2\mathcal{R}^{1/2}}{N^{1/2}}+MQ
              \nonumber \\
\ll & \,\,X^{19/11-12\varepsilon}\mathcal{R}^{-1}+X^{14/11+13\varepsilon}\mathcal{R}^{-1/2}
      +X^{7/11+11\varepsilon/2}\ll X.
\end{align*}
Writing $v_1=\min v_1(m),v_2=\max v_2(m)$, one has
\begin{align*}
\mathcal{S}(X,Q)
 & =\sum_{1\leqslant q\leqslant Q}\sum_{0\leqslant j\leqslant J_0}\sum_{M<m\leqslant2M}\sum_{v_1(m)<v\leqslant v_2(m)}b_v\frac{e\left(s(m,v)+1/8\right)}{\sqrt{|G(m,v)|}}+O(X)\\
 & \ll \sum_{1\leqslant q\leqslant Q}\sum_{0\leqslant j\leqslant J_0}\sum_{v_1<v\leqslant v_2}\left|\sum_{m\in I_v}\frac{e\left(s(m,v)\right)}{\sqrt{|G(m,v)|}}\right|+X,
\end{align*}
where $I_v$ is a subinterval of $[M,2M]$. Thus, we only need to estimate the sum over $m$. First, we shall prove that $|G(m,v)|^{-1/2}$ is monotonic. Differentiating the equation $f_n\left(m,g(m,v)\right)=v$ over $m$, we get
\begin{equation*}
f_{nm}(m,g)+f_{nn}(m,g)g_m(m,v)=0,
\end{equation*}
which implies that
\begin{equation}\label{gm(m,v)}
g_m(m,v)=-\frac{f_{nm}(m,g)}{f_{nn}(m,g)}.
\end{equation}
Putting (\ref{gm(m,v)}) into $G_m(m,v)$, we have
\begin{equation}\label{Gm(m,v)}
G_m(m,v)=f_{nnm}(m,g)+f_{nnn}(m,g)g_{m}(m,v)
=\frac{f_{nnm}f_{nn}-f_{nnn}f_{nm}}{f_{nn}}.
\end{equation}
We remark that the denominator $f_{nn}$ of (\ref{Gm(m,v)}) is always positive or negative. Therefore, we only need to consider $f_{nnm}f_{nn}-f_{nnn}f_{nm}$. From (\ref{Delta-n-q-gamma}), we have
\begin{align*}
f_{nm}(m,g)
 & =\gamma^2_1h_1m^{\gamma_1-1}\Delta(g,q;\gamma_1-1)+\gamma^2_2h_2m^{\gamma_2-1}\Delta(g,q;\gamma_2-1)\\
 & =\gamma^2_1(\gamma_1-1)h_1m^{\gamma_1-1}qg^{\gamma_1-2}+\gamma^2_2(\gamma_2-1)h_2m^{\gamma_2-1}qg^{\gamma_2-2}
 +O\left(\frac{q^2\mathcal{R}}{MN^3}\right)\\
 & =\left(\gamma^2_1(\gamma_1-1)h_1m^{\gamma_1-1}qg^{\gamma_1-2}+\gamma^2_2(\gamma_2-1)h_2m^{\gamma_2-1}qg^{\gamma_2-2}\right)
 \left(1+O\left(\frac{qQ}{N}\right)\right),
\end{align*}
where we use $|f_{nm}|>q\mathcal{R}Q^{-1}M^{-1}N^{-2}$.
Similarly, we have
\begin{align*}
 & f_{nn} = \left(c_{-2}(\gamma_1)h_1m^{\gamma_1}qg^{\gamma_1-3}+c_{-2}(\gamma_2)h_2m^{\gamma_2}qg^{\gamma_2-3}\right)
\left(1+O\left(\frac{qQ}{N}\right)\right),\\
 & f_{nnm}= \left(\gamma_1c_{-2}(\gamma_1)h_1m^{\gamma_1-1}qg^{\gamma_1-3}
+\gamma_2c_{-2}(\gamma_2)h_2m^{\gamma_2-1}qg^{\gamma_2-3}\right) \left(1+O\left(\frac{qQ}{N}\right)\right),\\
 & f_{nnn}= \left(c_{-3}(\gamma_1)h_1m^{\gamma_1}qg^{\gamma_1-4}
+c_{-3}(\gamma_2)h_2m^{\gamma_2}qg^{\gamma_2-4}\right) \left(1+O\left(\frac{qQ}{N}\right)\right),
\end{align*}
where $c_{-i}(\gamma)=\gamma(\gamma-1)(\gamma-2)\cdots(\gamma-i)$ for $i\geqslant 0$. For simplicity, we write $s=h_1m^{\gamma_1}g^{\gamma_1},\,t=h_2m^{\gamma_2}g^{\gamma_2}$. Then we get
\begin{equation*}
f_{nnm}f_{nn}-f_{nnn}f_{nm}=\frac{q}{mg^6}\left(As^2+2Bts+Ct^2\right)\left(1+O\left(\frac{qQ}{N}\right)\right),
\end{equation*}
where
\begin{align*}
 & A=\gamma^3_1(\gamma_1-1)^2(\gamma_1-2)<0,\quad C=\gamma^3_2(\gamma_2-1)^2(\gamma_2-2)<0,\\
 & B=\gamma_1(\gamma_1-1)\gamma_2(\gamma_2-1)(3\gamma_1\gamma_2-\gamma^2_1-\gamma^2_2-\gamma_1-\gamma_2)<0.
\end{align*}
We need to show that
\begin{equation}\label{A-B-C-Func}
As^2+2Bts+Ct^2\neq0.
\end{equation}
If $h_1h_2>0$, (\ref{A-B-C-Func}) is obvious. Now we assume that $h_1h_2<0$. It is easy to show that
\begin{equation*}
B^2-AC=\gamma^2_1(\gamma_1-1)^2\gamma^2_2(\gamma_2-1)^2(\gamma_1-\gamma_2)^2
(2\gamma_1+2\gamma_2+1+\gamma^2_1+\gamma^2_2-4\gamma_1\gamma_2)>0.
\end{equation*}
Thus there exist constants $a_1,a_2,b_1,b_2$ such that $As^2+2Bts+Ct^2=(a_1s+b_1t)(a_2s+b_2t)$.
Since $A<0,\,B<0,\,C<0$, it is easy to see that $a_1b_1>0,\,a_2b_2>0$. By noting that  $s$ and $t$ do not satisfy the condition of Case $0$, taking
\begin{equation*}
\eta=\frac{1}{2}\min\left(|a_1|,|a_2|,|b_1|^{-1},|b_2|^{-1}\right)
\end{equation*}
in Case $0$, we derive that
\begin{equation*}
|a_1s+b_1t|>\frac{\mathcal{R}}{Q^{1/2}\mathscr{L}},\qquad
|a_2s+b_2t|>\frac{\mathcal{R}}{Q^{1/2}\mathscr{L}}.
\end{equation*}
This is the reason why we consider Case $0$. Hence
\begin{equation*}
|As^2+2Bst+Ct^2|\geqslant\frac{\mathcal{R}^2}{Q\mathscr{L}^2}.
\end{equation*}
By the above arguments we know that $|G(m,v)|$ is  monotonic in $m$, so is $|G(m,v)|^{-1/2}$. Next, we shall compute $s_{mm}(m,v)$. Putting (\ref{gm(m,v)}) into $s_m(m,v)$ and $s_{mm}(m,v)$, we have
\begin{equation*}
  s_m(m,v)=f_m(m,g)+f_n(m,g)g_m-vg_m=f_m(m,g),
\end{equation*}
\begin{equation*}
  s_{mm}(m,v)=f_{mm}(m,g)+f_{mn}(m,g)g_m=\frac{f_{mm}f_{nn}-f_{mn}^2}{f_{nn}}.
\end{equation*}
Similar to $G_m$, we have
\begin{equation*}
  f_{mm}f_{nn}-f_{mn}^2=-\frac{2q^2}{m^2n^4}(A_1s^2+B_1st+C_1t^2)\left(1+O\left(\frac{Q^2}{N\mathscr{L}}\right)\right),
\end{equation*}
where
\begin{align*}
 A_1=\gamma_1^3(\gamma_1-1)^2,\qquad B_1=\gamma_1(\gamma_1-1)\gamma_2(\gamma_2-1)(\gamma_1+\gamma_2),  \qquad C_1=\gamma_2^3(\gamma_2-1)^2.
\end{align*}
Thus, we have $B_1^2-4A_1C_1 >0$. If $h_1h_2>0$, we immediately derive that
\begin{equation*}
|f_{mm}f_{nn}-f_{mn}^2|\gg\frac{q^2\mathcal{R}^2}{M^2N^4}.
\end{equation*}
If $h_1h_2<0$, similar to $G_m$, we have
\begin{equation*}
 |A_1s^2+B_1st+C_1t^2|\gg \frac{\mathcal{R}^2}{Q\mathscr{L}^2},
\end{equation*}
which implies that
\begin{equation*}
|f_{mm}f_{nn}-f_{mn}^2|\gg\frac{q^2\mathcal{R}^2}{M^2N^4Q\mathscr{L}^2}.
\end{equation*}
Combining the above arguments, we obtain
\begin{equation*}
|s_{mm}|\gg \frac{q\mathcal{R}}{M^2NQ\mathscr{L}^2}.
\end{equation*}
On the other hand, we have
\begin{equation*}
|s_{mm}|\ll|f_{mm}|+|f_{mn}g_m|\ll \frac{q\mathcal{R}}{M^2N}+\frac{q\mathcal{R}}{MN^2}\cdot\frac{N}{M}\ll\frac{q\mathcal{R}}{M^2N}.
\end{equation*}
Define
\begin{equation*}
T_{v,l}=\left\{m:m\in I_v,\,\frac{2^lq\mathcal{R}}{M^2NQ\mathscr{L}^2}<|s_{mm}|\leqslant\frac{2^{l+1}q\mathcal{R}}{M^2NQ\mathscr{L}^2}
 \ll \frac{q\mathcal{R}}{M^2N}\right\},
\end{equation*}
\begin{equation*}
 0\leqslant l\leqslant L:=\bigg\lfloor \frac{\log(Q\mathscr{L}^2)}{\log 2}\bigg\rfloor.
\end{equation*}
It follows from partial summation and (\ref{2-derivative}) of Lemma \ref{1-2-derivative} that
\begin{align*}
\mathcal{S}(X,Q)
\ll & \sum_{1\leqslant q\leqslant Q}\sum_{1\leqslant j\leqslant J_0}\sum_{v_1<v\leqslant v_2}\sum_{0\leqslant l\leqslant L}\left|\sum_{m\in T_{v,l}}\frac{e(s(m,v))}{\sqrt{|G(m,v)|}}\right|+X\\
\ll & \sum_{1\leqslant q\leqslant Q}\sum_{1\leqslant j\leqslant J_0}\sum_{v_1<v\leqslant v_2}\sum_{0\leqslant l\leqslant L}\left(\frac{N^3Q}{q\mathcal{R}}\right)^{\frac{1}{2}}\left(M\left(\frac{2^lq\mathcal{R}}{M^2NQ\mathscr{L}^2}\right)^{\frac{1}{2}}
+\left(\frac{2^lq\mathcal{R}}{M^2NQ\mathscr{L}^2}\right)^{-\frac{1}{2}}\right)+X\\
\ll & \sum_{1\leqslant q\leqslant Q}\sum_{1\leqslant j\leqslant J_0}\sum_{v_1<v\leqslant v_2}\sum_{0\leqslant l\leqslant L}(N\mathscr{L}^{-1}\sqrt{2}^l+MN^2Qq^{-1}\mathcal{R}^{-1}\mathscr{L}\sqrt{2}^{-l})+X\\
\ll & N^{-1}Q^{5/2}\mathcal{R}\mathscr{L}^2+MQ^2\mathscr{L}^2 +X\ll X^{1-\varepsilon}\mathscr{L}^2+X^{21/11-4\varepsilon}\mathcal{R}^{-1}+X\ll X.
\end{align*}
Combining the above five cases, we complete the proof of Lemma \ref{S(I)}. $\hfill$
\end{proof}

\begin{lemma}\label{Apply-H-B-identity}
Suppose that $x/2<X\leqslant x, X<X_1\leqslant2X$. Then we have
\begin{equation*}
T^*(X):=\sum_{X<n\leqslant X_1}\Lambda(n)e(h_1n^{\gamma_1}+h_2n^{\gamma_2})\ll X^{\gamma_1+\gamma_2-1-3\varepsilon}.
\end{equation*}
\end{lemma}
\begin{proof}
By Heath--Brown's identity with $k=4$, one can see that $T^*(X)$ can be written as linear combination of
$O(\log^8 X)$ sums, each of which is of the form
\begin{equation}\label{T(N1-N8)}
\mathscr{T}(X)=\sum_{n_1\sim N_1}\cdots\sum_{n_8\sim N_8}\log {n_1}\mu(n_5)\mu(n_6)\mu(n_7)\mu(n_8)e(h_1(n_1\dots n_8)^{\gamma_1}+h_2(n_1\dots n_8)^{\gamma_2}),
\end{equation}
where $X\ll N_1\cdots N_8\ll X$; $ N_i\leqslant (2X)^{1/4}, i=5,6,7,8$ and some $n_i$ may only take value $1$. Therefore, it is sufficient to give upper bound estimate for each $\mathscr{T}(X)$ defined as in (\ref{T(N1-N8)}). Next, we will consider three cases.

\noindent
\textbf{Case 1.} If there exists an $N_j$ such that $N_j\gg X^{-6/11+21\varepsilon}\mathcal{R}>X^{1/4}$, then we must have $1\leqslant j\leqslant 4$. Without loss of generality, we postulate that $N_j=N_1$, and take $m=n_2n_3\cdots n_8,\,\,n=n_1$. Trivially, there holds $m\ll X^{17/11-21\varepsilon}\mathcal{R}^{-1}$. Set
\begin{equation*}
a(m)=\sum_{m=n_2n_3\cdots n_8}\mu(n_5)\cdots\mu(n_8)\ll d_7(m).
\end{equation*}
Then $\mathscr{T}(X)$ is a sum of the form $S_I(M,N)$. By Lemma \ref{S(I)}, we have
\begin{equation*}
X^{-\varepsilon}\mathscr{T}(X)\ll X^{\gamma_1+\gamma_2-1-5\varepsilon}.
\end{equation*}

\noindent
\textbf{Case 2.}
If there exists an $N_j$ such that $X^{{2/11+8\varepsilon}}\ll N_j\ll X^{16/11-24\varepsilon}\mathcal{R}^{-1}$,
then we take
\begin{equation*}
n=n_j,\quad N=N_j,\qquad m=\prod_{i\neq j}n_i,\quad M=\prod_{i\neq j}N_i.
\end{equation*}
Thus, $\mathscr{T}(X)$ is a sum of the form $S_{II}(M,N)$ with $X^{{2/11+8\varepsilon}}\ll N\ll X^{16/11-24\varepsilon}\mathcal{R}^{-1}$. By Lemma \ref{S(II)}, we have
\begin{equation*}
X^{-\varepsilon}\mathscr{T}(X)\ll X^{\gamma_1+\gamma_2-1-5\varepsilon}.
\end{equation*}

\noindent
\textbf{Case 3.} If $N_j\ll X^{2/11+8\varepsilon}(j=1,2,\dots,8)$, without loss of generality, we postulate that $N_1\geqslant N_2\geqslant \dots\geqslant N_8$. Let $\ell$ be the natural number such that
\begin{equation*}
N_1\cdots N_{\ell-1}< X^{2/11+8\varepsilon}, \qquad N_1\cdots N_{\ell}\geqslant X^{2/11+8\varepsilon}.
\end{equation*}
It is easy to check that $2\leqslant \ell\leqslant7$. Then we have
\begin{equation*}
X^{2/11+8\varepsilon}<N_1\cdots N_{\ell}=N_1\cdots N_{\ell-1}N_{\ell}\leqslant \left(X^{2/11+8\varepsilon}\right)^2\leqslant X^{16/11-24\varepsilon}\mathcal{R}^{-1}.
\end{equation*}
In this case, we take
\begin{equation*}
m=\prod_{j=\ell+1}^8 n_j,\quad M=\prod_{j=\ell+1}^8 N_j,\qquad n=\prod_{j=1}^{\ell}n_j,\quad N=\prod_{j=1}^{\ell}n_j.
\end{equation*}
Then $\mathscr{T}(X)$ is a sum of the form $S_{II}(M,N)$. By Lemma \ref{S(II)}, we have
\begin{equation*}
X^{-\varepsilon}\mathscr{T}(X)\ll X^{\gamma_1+\gamma_2-1-5\varepsilon}.
\end{equation*}

\noindent
Combining the above three cases, we derive that
\begin{equation*}
T^*(X)\ll \mathscr{T}(X)\cdot \log^8X \ll X^{\gamma_1+\gamma_2-1-3\varepsilon}.
\end{equation*}
This completes the proof of Lemma \ref{Apply-H-B-identity}.   $\hfill$
\end{proof}

\section{Proof of Theorem 1.1}
For $1/2<\gamma<1$, it is easy to see that
\begin{equation*}
\lfloor-p^{\gamma}\rfloor-\lfloor-(p+1)^{\gamma}\rfloor=
\begin{cases}
     1,  & \textrm{if $p=\lfloor n^{1/\gamma}\rfloor$},\\
     0, & \textrm{otherwise},
   \end{cases}
\end{equation*}
and
\begin{equation}\label{Elementary-Formula}
(p+1)^{\gamma}-p^{\gamma}=\gamma p^{\gamma-1}+O(p^{\gamma-2}).
\end{equation}
Thus, we have
\begin{align*}
\pi(x;\gamma_1,\gamma_2)
&
=\sum_{p\leqslant x}\left(\lfloor-p^{\gamma_1}\rfloor-\lfloor-(p+1)^{\gamma_1}\rfloor\right)\left(\lfloor -p^{\gamma_2}\rfloor-\lfloor-(p+1)^{\gamma_2}\rfloor\right)\\
&
=\mathcal{F}_1(x)+O\left(|\mathcal{F}_2(x)|+|\mathcal{F}_3(x)|+|\mathcal{F}_4(x)|\right),
\end{align*}
where
\begin{align*}
&\mathcal{F}_1(x)=\sum_{p\leqslant x}\left((p+1)^{\gamma_1}-p^{\gamma_1}\right)\left((p+1)^{\gamma_2}-p^{\gamma_2}\right),\\
&\mathcal{F}_2(x)=\sum_{p\leqslant x}\left((p+1)^{\gamma_1}-p^{\gamma_1}\right)\left(\psi\left(-(p+1)^{\gamma_2}\right)-\psi\left(-p^{\gamma_2}\right)\right),\\
&\mathcal{F}_3(x)=\sum_{p\leqslant x}\left(\psi\left(-(p+1)^{\gamma_1}\right)-\psi\left(-p^{\gamma_1}\right)\right)\left((p+1)^{\gamma_2}-p^{\gamma_2}\right),\\
&\mathcal{F}_4(x)=\sum_{p\leqslant x}\left(\psi\left(-(p+1)^{\gamma_1}\right)-\psi\left(-p^{\gamma_1}\right)\right)\left(\psi\left(-(p+1)^{\gamma_2}\right)-\psi\left(-p^{\gamma_2}\right)\right).
\end{align*}
We shall deal with $\mathcal{F}_i(x)(i=1,2,3,4)$ separately. It follows from partial summation, (\ref{Elementary-Formula}) and prime number theorem that
\begin{align*}
\mathcal{F}_1(x)&=\sum_{p\leqslant x}\left(\gamma_1\gamma_2p^{\gamma_1+\gamma_2-2}+O\left(p^{\gamma_1+\gamma_2-3}\right)\right)\\
&=\gamma_1\gamma_2\int_{2}^x\frac{t^{\gamma_1+\gamma_2-2}}{\log t}\mathrm{d}t+O\big(x^{\gamma_1+\gamma_2-1}e^{-c_0\sqrt{\log x}}\big)
\end{align*}
for some $c_0>0$. By (\ref{Elementary-Formula}), partial summation and the arguments of Heath--Brown in \cite{Heath-Brown-1983}, we deduce that
\begin{align*}
\mathcal{F}_2(x)=\sum_{p\leqslant x}\left(\gamma_1p^{\gamma_1-1}+O\left(p^{\gamma_1-2}\right)\right)
\left(\psi\left(-(p+1)^{\gamma_2}\right)-\psi\left(-p^{\gamma_2}\right)\right)
\ll x^{\gamma_1+\gamma_2-1-\varepsilon}.
\end{align*}
Similarly, one also has $\mathcal{F}_3(x)\ll x^{\gamma_1+\gamma_2-1-\varepsilon}$. Now we focus on the upper bound estimate of $\mathcal{F}_4$. It is sufficient to show that, for any $x^{1/2}\ll X\ll x$, there holds
\begin{equation*}
T(X):=\sum_{X<n\leqslant 2X}\Lambda(n)\left(\psi\left(-(n+1)^{\gamma_1}\right)-\psi\left(-n^{\gamma_1}\right)\right)
\left(\psi\left(-(n+1)^{\gamma_2}\right)-\psi\left(-n^{\gamma_2}\right)\right)\ll X^{\gamma_1+\gamma_2-1-\varepsilon}.
\end{equation*}
Let $H_1=X^{1-\gamma_1+\varepsilon},\,\,H_2=X^{1-\gamma_2+\varepsilon}$. From Lemma \ref{Finite-Fourier-expansion}, we have
\begin{equation}\label{psi-psi}
\psi\left(-(n+1)^{\gamma_i}\right)-\psi\left(-n^{\gamma_i}\right)=M_{H_i}(n)+E_{H_i}(n),\qquad (i=1,2),
\end{equation}
where
\begin{align}
& M_{H_i}(n)=-\sum_{1\leqslant |h_i|\leqslant H_i}\frac{e\left(-h_i(n+1)^{\gamma_i}\right)-e\left(-h_in^{\gamma_i}\right)}{2\pi ih_i},\label{MH(n)}\\
& E_{H_i}(n)=O\left(\min\left(1,\frac{1}{H_i\|(n+1)^{\gamma_i}\|}\right)\right)
+O\left(\min\left(1,\frac{1}{H_i\|n^{\gamma_i}\|}\right)\right).\label{EH(n)}
\end{align}
Putting (\ref{psi-psi}) into $T(X)$, we obtain
\begin{align*}
T(X) & =\sum_{X<n\leqslant2X}\Lambda(n)(M_{H_1}(n)+E_{H_1}(n))(M_{H_2}(n)+E_{H_2}(n))\\
 & =:T_1(X)+T_2(X)+T_3(X)+T_4(X),
\end{align*}
where
\begin{align*}
& T_1(X)=\sum_{X<n\leqslant2X}\Lambda(n)M_{H_1}(n)M_{H_2}(n),\qquad T_2(X)=\sum_{X<n\leqslant2X}\Lambda(n)M_{H_1}(n)E_{H_2}(n),\\
&
T_3(X)=\sum_{X<n\leqslant2X}\Lambda(n)E_{H_1}(n)M_{H_2}(n),\qquad T_4(X)=\sum_{X<n\leqslant2X}\Lambda(n)E_{H_1}(n)E_{H_2}(n).
\end{align*}
For each fixed $n$, let
\begin{align*}
& \phi_n(t)=t^{-1}\left(e\left(t(n+1)^{\gamma_1}-tn^{\gamma_1}\right)-1\right),\\
& S_n(t)=\sum_{1\leqslant h\leqslant t}e(hn^{\gamma_1}).
\end{align*}
It is easy to check that
\begin{align*}
\phi_n(t)\ll n^{\gamma_1-1},\qquad\frac{\partial\phi_n(t)}{\partial t}\ll t^{-1}n^{\gamma_1-1},\qquad
 S_n(t)\ll \min\left(t,\frac{1}{\|n^{\gamma_1}\|}\right).
\end{align*}
By partial summation, we obtain
\begin{align}\label{Esti-MH(n)}
|M_{H_1}(n)|
\leqslant & \,\,\left|\sum_{1\leqslant h\leqslant H_1}e\left(hn^{\gamma_1}\right)\phi_n(h)\right|\ll\left|\int_{1}^{H_1}\phi_n(t)\mathrm{d}S_n(t)\right|
                    \nonumber \\
\ll &\,\, \big|\phi_n(H_1)\big|\big|S_n(H_1)\big|+\int_{1}^{H_1}\big|S_n(t)\big|\bigg|\frac{\partial\phi_n(t)}{\partial t}\bigg|\mathrm{d}t
                    \nonumber \\
\ll &\,\,  H_1X^{\gamma_1-1}\log X\cdot\min\left(1,\frac{1}{H_1\|n^{\gamma_1}\|}\right).
\end{align}
By (\ref{EH(n)}), (\ref{Esti-MH(n)})  and  Lemma \ref{k-min-esti} with $k=2$, we obtain
\begin{align*}
T_2(X)+T_3(X)+T_4(X) & \ll \mathscr{L}\sum_{X<n\leqslant 2X}\left(|M_{H_1}(n)E_{H_2}(n)|+|E_{H_1}(n)M_{H_2}(n)|+|E_{H_1}(n)E_{H_2}(n)|\right)\\
& \ll X^{\varepsilon}\sup_{(u_1,u_2)\in[0,1]^2}\sum_{X<n\leqslant2X}\prod_{i=1}^2\min\left(1,\frac{1}{H_i\|(n+u_i)^{\gamma_i}\|}\right)\\
& \ll X^{\varepsilon}\left(\frac{X(\log X)^2}{H_1H_2}+X^{2/3}(\log X)^2\right)\ll X^{\gamma_1+\gamma_2-1-2\varepsilon}+X^{2/3+2\varepsilon/3},
\end{align*}
provided that $\gamma_1+\gamma_2>5/3$. Next, we shall give the upper bound estimate of $T_1(X)$. Define
\begin{align*}
 \Psi_{h,\gamma}(n)=e\left(h(n+1)^{\gamma}-hn^{\gamma}\right)-1,\qquad
 \Psi(n)=\Psi_{h_1,\gamma_1}(n)\Psi_{h_2,\gamma_2}(n).
\end{align*}
It is easy to check that
\begin{align*}
\Psi(n)\ll |h_1h_2|n^{\gamma_1+\gamma_2-2},\qquad\qquad \frac{\partial\Psi(n)}{\partial n}\ll |h_1h_2|n^{\gamma_1+\gamma_2-3}.
\end{align*}
Therefore, by partial summation, (\ref{MH(n)}) and Lemma \ref{Apply-H-B-identity}, we derive that
\begin{align*}
T_1(X)
 & \ll \sum_{1\leqslant|h_1|\leqslant H_1}\sum_{1\leqslant|h_2|\leqslant H_2}\frac{1}{|h_1h_2|}\left|\sum_{X<n\leqslant2X}\Lambda(n)e(h_1n^{\gamma_1}+h_2n^{\gamma_2})\Psi(n)\right|\\
 & \ll \sum_{1\leqslant|h_1|\leqslant H_1}\sum_{1\leqslant|h_2|\leqslant H_2}\frac{1}{|h_1h_2|}\left|\int_{X}^{2X}\Psi(t)\mathrm{d}\Bigg(\sum_{X<n\leqslant t}\Lambda(n)e(h_1n^{\gamma_1}+h_2n^{\gamma_2})\Bigg)\right|\\
 & \ll \sum_{1\leqslant|h_1|\leqslant H_1}\sum_{1\leqslant|h_2|\leqslant H_2}\frac{1}{|h_1h_2|}\big|\Psi(2X)\big|\left|\sum_{X<n\leqslant2X}\Lambda(n)e(h_1n^{\gamma_1}+h_2n^{\gamma_2})\right|\\
 &\quad+\sum_{1\leqslant|h_1|\leqslant H_1}\sum_{1\leqslant|h_2|\leqslant H_2}\int_{X}^{2X}\bigg|\frac{\partial{\Psi(t)}}{\partial t}\bigg|
     \Bigg|\sum_{X<n\leqslant t}\Lambda(n)e(h_1n^{\gamma_1}+h_2n^{\gamma_2})\Bigg|\mathrm{d}t \\
 & \ll X^{\gamma_1+\gamma_2-2}\max_{\substack{1\leqslant|h_1|\leqslant H_1\\1\leqslant|h_2|\leqslant H_2}}\max_{X<t\leqslant2X}\sum_{1\leqslant|h|\leqslant H_1}\sum_{1\leqslant|h|\leqslant H_2}\left|\sum_{X<n\leqslant t}\Lambda(n)e(h_1n^{\gamma_1}+h_2n^{\gamma_2})\right|\\
 & \ll X^{\gamma_1+\gamma_2-1-\varepsilon}.
\end{align*}
This completes the proof of Theorem \ref{Theorem}.

\section*{Acknowledgement}

The authors would like to appreciate the referee for his/her patience in refereeing this paper.
This work is supported by the National Natural Science Foundation of China (Grant Nos. 11971476, 11901566, 12001047, 12071238).

\end{document}